\documentclass[11pt]{article}
\usepackage[margin=1in]{geometry}
\usepackage{amsmath,amsthm,amssymb,amsfonts,amscd}
\usepackage{graphicx}
\usepackage{color}
\usepackage{tikz}

\newtheorem{theorem}{Theorem}
\newtheorem{proposition}{Proposition}
\newtheorem{algorithm}{Algorithm}
\newtheorem{corollary}{Corollary}
\newtheorem{definition}{Definition}
\newtheorem{example}{Example}
\newtheorem{remark}{Remark}

\newtheorem{keywords}{Keywords}
\DeclareMathOperator{\Ima}{im}

\begin{document}

\title{Bisimulation equivalence of DAE systems}

\author{Noorma Yulia Megawati\textsuperscript{a,b}$^{\ast}$
		and Arjan van der Schaft\textsuperscript{a}\\ 
		}
		\thanks{Corresponding author. Email: n.y.megawati@rug.nl; noorma\_yulia@ugm.ac.id\\$^a$Johann Bernoulli Institute for Mathematics and Computer Science,
			University of Groningen, 9700AK Groningen, The Netherlands\\
			$^b$Department of Mathematics, Universitas Gadjah Mada, Yogyakarta, Indonesia}
	
\maketitle

\begin{abstract}
In this paper the notion of bisimulation relation for linear input-state-output systems is extended to general linear differential-algebraic (DAE) systems. Geometric control theory is used to derive a linear-algebraic characterization of bisimulation relations, and an algorithm for computing the maximal bisimulation relation between two linear DAE systems. The general definition is specialized to the case where the matrix pencil $sE - A$ is regular. Furthermore, by developing a one-sided version of bisimulation, characterizations of simulation and abstraction are obtained.
\end{abstract}

\begin{keywords}
Differential-algebraic system, bisimulation, consistent subset,  regular pencil, abstraction, maximal bisimulation relation.
\end{keywords}

\section{Introduction}
A fundamental concept in the broad area of systems theory, concurrent processes, and dynamical systems, is the notion of {\it equivalence}. In general there are different ways to describe systems (or, processes); each with their on advantages and possibly disadvantages. This call for systematic ways to convert one representation into another, and for means to determine which system representations are 'equal'. It also involves the notion of {\it minimal} system representation.

Furthermore, in systems theory and the theory of concurrent processes, the emphasis is on determining which systems are {\it externally equivalent}; we only want to distinguish between systems if the distinction can be detected by an external system interacting with these systems. This is crucial in any modular approach to the control and design of complex systems.

Classical notions developed in systems and control theory for external equivalence are {\it transfer matrix equality} and {\it state space equivalence}. Within computer science the basic notion has been called {\it bisimulation relation} \cite{Cla99}. An extension of the notion of bisimulation to continuous dynamical systems has been explored before in a series of innovative papers by Pappas and co-authors \cite{Pap03,Tab04}. More recently, motivated by the rise of hybrid and cyber-physical systems, a reapproachment of these notions stemming from different backgrounds has been initiated. In particular, it has been shown how for linear systems a notion of bisimulation relation can be developed mimicking the notion of bisimulation relation for transition systems, and directly extending classical notions of transfer matrix equality and state space equivalence \cite{Arjan04a}. An important aspect of this approach in developing bisimulation theory for continuous linear systems is that the conditions for existence of a bisimulation relation are formulated directly in terms of the differential equation description, instead of the corresponding dynamical behavior (the solution set of the differential equations). This has dramatic consequences for the complexity of bisimulation computations, which reduce to linear-algebraic computations on the matrices specifying the linear system descriptions, very much in the spirit of linear geometric control theory \cite{Won74,Bas92}. For extensions to nonlinear systems exploiting corresponding nonlinear geometric theory we refer to \cite{Arjan04a}.

The present paper continues on these developments by extending the notion of bisimulation relation to general linear {\it differential-algebraic} (DAE) systems involving disturbances (capturing non-determinism). This is well-motivated since complex system descriptions usually arise from interconnection of system components, and generally lead to descriptions involving both differential equations and algebraic equations. Indeed, {\it network modelling} almost invariably leads to differential-algebraic systems. The aim of this paper is to determine linear-algebraic conditions for the existence of a bisimulation relation, directly in terms of the differential-algebraic equations instead of computing the solution trajectories. 

As in previous work on bisimulation theory for input-state-output systems \cite{Arjan04b}, we explicitly allow for the possibility of 'non-determinism' in the sense that the state may evolve according to different time-trajectories for the same values of the external variables. This 'non-determinism' may be explicitly modeled by the presence of internal 'disturbances' or implicitly by non-uniqueness of the solutions of differential-algebraic equations. Non-determinism may be an intrinsic feature of the system representation (as due e.g. to non-uniqueness of variables in the internal subsystem interconnections), but may also arise by {\it abstraction} of the system to a lower-dimensional system representation. By itself, the notion of abstraction can be covered by a one-way version of bisimulation, called {\it simulation}, as will be discussed in Section 5.

The structure of this paper is as follows. In Section 2 we provide the theory concerning differential-algebraic equation (DAE) systems which will be used in the sequel. These DAE systems are given in {\it descriptor system} format $E\dot{x} = Ax + Bu + Gd, y=Cx$, with $u,y$ being the external variables (inputs and outputs), $d$ the disturbances modelling internal non-determinism, and $x$ the (not necessarily minimal) state. In Section 3 we give the definition of bisimulation relation for DAE systems, and a full linear-algebraic characterization of them, together with a geometric algorithm to compute the maximal bisimulation relation between two linear systems. In Section 4 we study the implication of adding the condition of {\it regularity} to the matrix pencil $sE -A$, and show how in this case bisimilarity reduces to equality of transfer matrices. Finally, simulation relations and the accompanying notion of abstraction are discussed in Section 5.

%%%%%%%%%%%%%%%%%%%%%%%%%%%%%%%%%%%%%%%%%%%%%%%%%%%%%%%%%%%%%%%%%%%%%%%%%%%%%%%%  
\section{Preliminaries on linear DAE systems}
In this paper we consider the following general class of linear differential-algebraic (DAE) systems
\begin{equation}
\label{1}
\Sigma:\begin{array}{rcll}
E\dot{x} & =& Ax+Bu+Gd,  \quad & x\in\mathcal{X}, u\in\mathcal{U}, d \in\mathcal{D} \\[2mm]
y& =& Cx,  \quad & y\in\mathcal{Y},
\end{array}
\end{equation}
where $E,A\in \mathbb{R}^{q\times n}$ and $B\in\mathbb{R}^{q\times m},G \in\mathbb{R}^{q\times s}, C\in \mathbb{R}^{p \times n}$; $\mathcal{X}, \mathcal{U}, \mathcal{D}$ and $\mathcal{Y}$ are finite dimensional linear spaces, of dimension, respectively, $n,m,s,p$. Here, $x$ denotes the state of the system (possibly constrained by linear equations), $u$ the input, $y$ the output and $d$ the 'disturbance' acting on the system. Furthermore, $q$ denotes the total number of (differential and algebraic) equations describing the dynamics of the system.
The allowed time-functions $x:\mathbb{R}^{+}\rightarrow \mathcal{X}$, $u:\mathbb{R}^{+}\rightarrow \mathcal{U}$, $y: \mathbb{R}^{+}\rightarrow \mathcal{Y}$, $d : \mathbb{R}^{+}\rightarrow \mathcal{D}$, with $\mathbb{R}^+=[0,\infty)$, will be denoted by $\mathfrak{X},\mathfrak{U},\mathfrak{Y},\mathfrak{D}$. The exact choice of function classes is for purposes of this paper not really important, as long as the state trajectories $x(\cdot)$ are at least continuous. For convenience, we will take $\mathfrak{U},\mathfrak{D}$ to be the class of piecewise-continuous and $\mathfrak{X}, \mathfrak{Y}$ the class of continuous and piecewise-differentiable functions on $\mathbb{R}^{+}$.  We will denote these functions by $x(\cdot),u(\cdot),y(\cdot),d(\cdot)$, and if no confusion can arise simply by $x,u,y,d$. We will primarily regard $d$ as an internal generator of 'non-determinism': multiple state trajectories may occur for the same initial condition $x(0)$ and input function $u(\cdot)$. This, for example, occurs by {\it abstracting} a deterministic system; see the developments in Section 5.

The {\it consistent subset} $\mathcal{V}^{*}$ for a system $\Sigma$ is given as the maximal subspace $\mathcal{V} \subset\mathbb{R}^{n}$ satisfying
\begin{equation}
\label{2}
\begin{array}{lllll}
(i) & A\mathcal{V}\subset E\mathcal{V}+\mathcal{G}\\[2mm]
(ii) & \Ima B \subset E\mathcal{V}+\mathcal{G}
\end{array}
\end{equation}
where $\mathcal{G}=\Ima G$, or is {\it empty} in case there does not exist any subspace $\mathcal{V}$ satisfying \eqref{2}. It follows that $\mathcal{V}^{*}$ equals the set of all initial conditions $x_{0}$ for which for every piecewise-continuous input function $u(\cdot)$ there exist a piecewise-continuous function $d(\cdot)$ and a continuous and piecewise-differentiable solution trajectory $x(\cdot)$ of $\Sigma$ with $x(0)=x_0$. 

\begin{remark}
	The definition of consistent subset $\mathcal{V}^{*}$ as given above {\it extends} the standard definition given in the literature on linear DAE and descriptor systems, see e.g. \cite{Ber13}. In fact, the above definition reduces to the definition in \cite{Ber13} for the case $B=0$ when additionally {\it renaming} the disturbance $d$ by $u$. (Thus in the standard definition the consistent subset is the set of initial conditions for which there exists an {\it input} function $u$ and a corresponding solution of the DAE with $d=0$.) 
	This extended definition of consistent subset, as well as the change in terminology between $u$ and $d$, is directly motivated by the notion of bisimulation where we wish to consider solutions of the system for {\it arbitrary} external input functions $u(\cdot)$; see also the definition of bisimulation for labelled transition systems \cite{Cla99}. Note that for $B = 0$ or void the zero subspace $\mathcal{V} = \{0\}$ always satisfies \eqref{2}, and thus $\mathcal{V}^*$ is a subspace. However for $B \neq 0$  there may {\it not} exist any subspace $\mathcal{V}$ satisfying (\ref{2}) in which case the consistent subset is {\it empty} (and thus strictly speaking not a subspace). In the latter case, such a system has empty input-output behavior from a bisimulation point of view.
\end{remark}
\begin{remark}
	Note that we {\it can} accommodate for additional restrictions on the allowed values of the input functions $u$, depending on the initial state, by making use of the following standard construction, incorporating $u$ into an extended state vector. Rewrite system (\ref{1}) as
	\begin{equation}
	\label{3}
	\Sigma_e:\begin{array}{rllll}
	\left[
	E \quad 0
	\right]\left[
	\begin{array}{cccc}
	\dot{x}\\\dot{u}
	\end{array}
	\right]&=& \left[
	A \quad B
	\right]\left[
	\begin{array}{cccc}
	x\\ u
	\end{array}
	\right]+Gd\\[3mm]
	y&=& \left[
	\begin{array}{cccc}
	C & 0
	\end{array}
	\right]\left[
	\begin{array}{cccc}
	x\\ u
	\end{array}
	\right]
	\end{array}
	\end{equation}
	Denote by $x_e=\left[
	\begin{array}{cccc}
	x\\u
	\end{array}
	\right]$ the {\it extended} state vector, and define $E_e := \begin{bmatrix} E & 0 \end{bmatrix}, A_e := \begin{bmatrix} A & B \end{bmatrix}$. Then the consistent \textit{subspace} $\mathcal{V}_{e}^{*}$ of system (\ref{3}) is given by the maximal subspace $\mathcal{V}_{e}\subset \mathcal{X}\times\mathcal{U}$ satisfying
	\begin{equation}
	\label{4}
	A_{e}\mathcal{V}_{e}\subset E_{e}\mathcal{V}_{e}+\mathcal{G}\\
	\end{equation}
	It can be easily seen that $\mathcal{V}^{*}\subset \pi_{x}(\mathcal{V}_{e}^*)$, where $\pi_{x}$ is the canonical projection of $\mathcal{X} \times \mathcal{U}$ on $\mathcal{X}$. The case $\mathcal{V}^{*}\subsetneq \pi_{x}(\mathcal{V}_{e}^*)$ corresponds to the presence of initial conditions which are consistent only for input functions taking value in a strict {\it subspace} of $\mathcal{U}$.
\end{remark}

In order to analyze the solutions of the linear DAE \eqref{1}, an important observation is that we can always {\it eliminate} the disturbances $d$. Indeed, given \eqref{1} we can construct matrices $G^{\perp}, G^{\dagger}$ and an $q \times q$ matrix $P$ such that
\begin{equation}
\label{5}
\begin{array}{cccccccc}
G^{\perp}G=0, & G^{\dagger}G=I_{s}, & P=\begin{bmatrix} G^{\perp} \\ G^{\dagger}\end{bmatrix} , &\text{rank}(P)=q
\end{array}
\end{equation}
( $G^{\perp}$ is a left annihilator of $G$ of maximal rank, and $G^{\dagger}$ is a left inverse of $G$.) By premultiplying both sides of (\ref{1}) by the invertible matrix $P$ it follows \cite{Kar81} that system (\ref{1}) is equivalent to
\begin{equation}
\label{6}
\begin{array}{rlll}
G^{\perp}E\dot{x} &=& G^{\perp}Ax+G^{\perp}Bu\\ [2mm]
d&=& G^{\dagger}(E\dot{x}-Ax-Bu)\\ [2mm]
y &=& Cx
\end{array}
\end{equation}
Hence the disturbance $d$ is specified by the second line of \eqref{6}, and the solutions $u(\cdot),x(\cdot)$ are determined by the first line of \eqref{6} not involving $d$. We thus conclude that for the theoretical study of the state trajectories $x(\cdot)$ corresponding to input functions $u(\cdot)$ we can always, without loss of generality, restrict attention to linear DAE systems of the form
\begin{equation}
\label{7}
\begin{array}{rllll}
E\dot{x} &=& Ax+Bu\\ [2mm]
y &=& Cx
\end{array}
\end{equation}
On the other hand, for computational purposes it is usually not desirable to eliminate $d$, since this will often complicate the computations and result in loss of insight into the model.

The next important observation is that for theoretical analysis any linear DAE system \eqref{7} can be assumed to be in the following {\it special form}, again without loss of generality. Take invertible matrices $S\in\mathbb{R}^{q\times q}$ and $T\in\mathbb{R}^{n \times n}$ such that
\begin{equation}\label{8}
SET=\left[\begin{array}{cccc}
I & 0\\0 & 0
\end{array}
\right]
\end{equation}
where the dimension $n_a$ of the identity block $I$ is equal to the rank of $E$. Split the transformed state vector $T^{-1}x$ correspondingly as $T^{-1}x=\left[\begin{array}{cccc}
x^{a}\\ x^{b}
\end{array}\right]$, with $\dim x^a=n_a, \dim x^b=n_b, n_a + n_b=n$. It follows that by premultiplying the linear DAE \eqref{7} by $S$ it transforms into an equivalent system (in the new state vector $T^{-1}x$) of the form
\begin{equation}\label{9}
\begin{array}{rcl}
\begin{bmatrix} 		
\dot{x}^{a}\\[2mm] 0 \end{bmatrix} & = & \begin{bmatrix} A^{aa} & A^{ab}\\[2mm] A^{ba} & A^{bb} \end{bmatrix}
\begin{bmatrix} x^{a}\\[2mm] x^{b}\end{bmatrix} + \begin{bmatrix} B^{a}\\[2mm] B^{b} \end{bmatrix}u \\ [5mm]
y & = & \begin{bmatrix} C^{a} & C^{b} \end{bmatrix} \begin{bmatrix} x^{a}\\[2mm] x^{b}\end{bmatrix}
\end{array}
\end{equation}
One of the advantages of the special form \eqref{9} is that the consistent subset $\mathcal{V}^*$ can be explicitly characterized using geometric control theory.
\begin{proposition}\label{Vstar}
	The set $\mathcal{V}^{*}$ of consistent states of (\ref{9}) is non-empty if and only if $B^b=0$ and $\Ima B^a \subset \mathcal{W}(A^{aa},A^{ab},A^{ba})$, where $\mathcal{W}(A^{aa},A^{ab},A^{ba})$ denotes the {\it maximal controlled invariant subspace} of the auxiliary system
	\begin{equation}\label{10}
	\begin{array}{rcl}
	\dot{x}^a & = &A^{aa}x^a + A^{ab}v \\
	w & = & A^{ba}x^a
	\end{array}
	\end{equation}
	with state $x^a$, input $v$, and output $w$. Furthermore, in case $\mathcal{V}^{*}$ is non-empty it is given by the subspace
	\begin{equation}
	\label{11}
	\begin{array}{rllll}
	\mathcal{V}^{*}&=&\{\left[
	\begin{array}{cccc}
	x^{a}\\x^{b}
	\end{array}
	\right]\mid x^{a}\in \mathcal{W}, x^{b}=Fx^{a}+z,\\ [2mm] && z\in\ker A^{bb}\cap (A^{ab})^{-1}\mathcal{W}(A^{aa},A^{ab},A^{ba})\}
	\end{array}
	\end{equation}
	where $(A^{ab})^{-1}$ denotes set-theoretic inverse, and where the matrix $F$ is a friend of $\mathcal{W}(A^{aa},A^{ab},A^{ba})$, i.e., 
	\begin{equation}\label{12}
	(A^{aa} + A^{ab}F)\mathcal{W}(A^{aa},A^{ab},A^{ba}) \subset \mathcal{W}(A^{aa},A^{ab},A^{ba})
	\end{equation}
\end{proposition}
\begin{proof}
	The first claim follows from the fact that the subset $\mathcal{V}^*$ of consistent states for \eqref{7} is non-empty if and only if, see \eqref{2}, $\Ima B \subset E\mathcal{V}^*$. The characterization of $\mathcal{V}^*$ given in \eqref{11} follows from the characterization of the maximal controlled invariant subspace of a linear system with feedthrough term as given e.g. in \cite[Theorem 7.11]{Trent01}.
\end{proof}
\begin{remark}
	The characterization of the consistent subspace $\mathcal{V}^*$ given in \eqref{11}, although being a direct consequence of geometric control theory, seems relatively unknown within the literature on DAE systems.
\end{remark}
\begin{remark}
	Usually, the maximal controlled invariant subspace is denoted by $\mathcal{V}^*(A^{aa},A^{ab},A^{ba})$; see e.g. \cite{Trent01}. However, in order to distinguish it from the consistent subset $\mathcal{V}^*$ we have chosen the notation $\mathcal{W}(A^{aa},A^{ab},A^{ba})$. In the rest of the paper we will abbreviate this, if no confusion is possible, to $\mathcal{W}$.
\end{remark}

Based on Proposition \ref{Vstar} we derive the following fundamental statement regarding solutions of linear DAE systems.
\begin{theorem}\label{thm5}
	Consider the linear DAE system (\ref{7}), with $\Ima B\subset E\mathcal{V}^{*}$. Then for all $u(\cdot)\in\mathfrak{U}$ continuous at $t=0$ and for all $x_{0}\in\mathcal{V}^{*}$ and $f\in\mathcal{V}^{*}$ satisfying
	\begin{equation}
	\label{13}
	Ef=Ax_{0}+Bu(0)
	\end{equation}
	there exists a continuous and piecewise-differentiable solution $x(\cdot)$ of (\ref{7}) satisfying
	\begin{equation}
	\label{14}
	\begin{array}{cccc}
	x(0)= x_{0} , & \dot{x}(0)=f .
	\end{array}
	\end{equation}
	Conversely, for all  $u(\cdot)\in\mathfrak{U}$ every continuous and piecewise-differentiable solution $x(\cdot)$ of (\ref{7}) which is differentiable at $t=0$ defines by (\ref{14}) $x_{0},f\in\mathcal{V}^{*}$ satisfying (\ref{13}).
\end{theorem}
\begin{proof}
	The last statement is trivial. Indeed, if $x(\cdot)$ is a differentiable solution of $E\dot{x}=Ax+Bu$ then $x(t)\in\mathcal{V}^{*}$ for all $t$, and thus $x(0)\in\mathcal{V}^{*}$ and by linearity $\dot{x}(0)\in\mathcal{V}^{*}$. Furthermore, $E\dot{x}(0)=Ax(0)+Bu(0)$.
	
	For the first claim, take $u(\cdot)\in\mathfrak{U}$ and consider any $x_{0},f\in\mathcal{V}^{*}$ satisfying (\ref{13}). As noted above we can assume that the system is in the form \eqref{9}. Then by (\ref{11})
	\begin{equation}
	\label{15}
	\begin{array}{rllll}
	x_{0}&=&\left[
	\begin{array}{cccc}
	x_{0}^{a}\\x_{0}^{b}
	\end{array}
	\right], \quad x_{0}^{a}\in\mathcal{W}, x_{0}^{b}=Fx_{0}^{a}+z_{0}, z_{0}\in\ker A^{bb}\cap (A^{ab})^{-1}\mathcal{W}\\ [5mm]
	f&=&\left[
	\begin{array}{cccc}
	f^{a}\\f^{b}
	\end{array}
	\right], \quad f^{a}\in\mathcal{W}, f^{b}=Ff^{a}+z_{f},z_{f}\in\ker A^{bb}\cap (A^{ab})^{-1}\mathcal{W}
	\end{array}
	\end{equation}
	Then consider the unique solution $x^{a}(\cdot)$ of
	\begin{equation}
	\label{16}
	\dot{x}^{a}=A^{aa}x^{a}+A^{ab}(Fx^{a}+z)+B^{a}u,\quad x^{a}(0)=x_{0}^{a}
	\end{equation}
	where the constant vector $z$ is chosen such that
	\begin{equation}
	\label{17}
	A^{aa}x_{0}^{a}+A^{ab}(Fx_{0}^{a}+z)+B^{a}u(0)=f^{a}.
	\end{equation}
	Furthermore, define the time-function
	\begin{equation}
	\label{18}
	x^{b}(t)=Fx^{a}(t)+z_{0}+tz_{f}
	\end{equation}
	Then by construction
	\begin{equation}
	\label{19}
	x(0)=\left[
	\begin{array}{cccc}
	x^{a}(0)\\ x^{b}(0)
	\end{array}
	\right]=\left[
	\begin{array}{cccc}
	x_{0}^{a}\\ Fx_{0}^{a}+z_{0}
	\end{array}
	\right]=x_{0}
	\end{equation}
	while
	\[
	\left[
	\begin{array}{cccc}
	\dot{x}^{a}(0)\\ \dot{x}^{b}(0)
	\end{array}
	\right]=\left[
	\begin{array}{cccc}
	A^{aa}x_{0}^{a}+A^{ab}(Fx_{0}^{a}+z)+B^{a}u(0)\\ F\dot{x}^{a}(0)+z_{f}
	\end{array}
	\right]=\left[
	\begin{array}{ccccc}
	f^{a}\\ Ff^{a}+z_{f}
	\end{array}
	\right]=\left[
	\begin{array}{cccc}
	f^{a}\\ f^{b}
	\end{array}
	\right].
	\]
\end{proof}
By recalling the equivalence between systems with disturbances \eqref{1} with systems without disturbances \eqref{7} we obtain the following corollary.		
\begin{corollary}
	Consider the linear DAE system (\ref{1}), with $\Ima B\subset E\mathcal{V}^{*} + \mathcal{G}$. Then for all $u(\cdot)\in\mathfrak{U}$, $d(\cdot)\in\mathfrak{D}$, continuous at $t=0$, and for all $x_{0}\in\mathcal{V}^{*}$ and $f\in\mathcal{V}^{*}$ satisfying
	\begin{equation}
	\label{20}
	Ef=Ax_{0}+Bu(0)+Gd(0)
	\end{equation}
	there exists a continuous and piecewise-differentiable solution $x(\cdot)$ of (\ref{1}) satisfying
	\begin{equation}
	\label{21}
	\begin{array}{cccc}
	x(0)= x_{0} , & \dot{x}(0)=f .
	\end{array}
	\end{equation}
	Conversely, for all  $u(\cdot)\in\mathfrak{U}, d(\cdot)\in\mathfrak{D}$ every continuous and piecewise-differentiable solution $x(\cdot)$ of (\ref{1}) which is differentiable at $t=0$ defines by (\ref{21}) $x_{0},f\in\mathcal{V}^{*}$ satisfying (\ref{20}).
\end{corollary}

%%%%%%%%%%%%%%%%%%%%%%%%%%%%%%%%%%%%%%%%%%%%%%%%%%%%%%%%%%%%%%%%%%%%%%%%%%%%%%%%
\section{Bisimulation relations for linear DAE systems}
Now, let us consider two systems of the form \eqref{1}
\begin{equation}
\label{22}
\Sigma_{i}:\begin{array}{rlllll}
E_{i}\dot{x}_{i} =& A_{i}x_{i}+B_{i}u_{i}+G_{i}d_{i}, & x_{i}\in\mathcal{X}_{i}, u_{i}\in\mathcal{U}, d_{i}\in\mathcal{D}_{i}\\[2mm]
y_{i}=& C_{i}x_{i}, & y_{i}\in\mathcal{Y}, \quad i=1,2.
\end{array}
\end{equation}
where $E_{i},A_{i}\in \mathbb{R}^{q_i\times n_i}$ and $B_{i}\in\mathbb{R}^{q_i\times m},G_{i}\in\mathbb{R}^{q_i\times s_i}, C_{i}\in \mathbb{R}^{p \times n_i}$ for $i=1,2,$ with $\mathcal{X}_{i},  \mathcal{D}_{i}, i=1,2,$ the state space and disturbance spaces, and $\mathcal{U}, \mathcal{Y}$ the common input and output spaces. The fundamental definition of bisimulation relation is given as follows.	
\begin{definition}\label{def1}
	A subspace
	\[
	\mathcal{R}\subset\mathcal{X}_{1}\times\mathcal{X}_{2},
	\]
	with $\pi_{i}(\mathcal{R})\subset \mathcal{V}_{i}^{*},$ where $\pi_{i}:\mathcal{X}_{1}\times\mathcal{X}_{2}\rightarrow \mathcal{X}_{i}$ denote the canonical projections for $i=1,2$, is a {\it bisimulation relation} between two systems $\Sigma_{1}$ and $\Sigma_{2}$ with consistent subsets $\mathcal{V}_{i}^{*},i=1,2,$ if and only if for all pairs of initial conditions $(x_{1},x_{2})\in\mathcal{R}$ and any joint input function $u_1(\cdot) = u_2(\cdot)=u(\cdot) \in \mathfrak{U}$ the following properties hold:
	\begin{enumerate}
		\item for every disturbance function $d_{1}(\cdot) \in \mathfrak{D}_1$ for which there exists a solution $x_{1}(\cdot)$ of $\Sigma_1$ (with $x_1(0)=x_{1}$), there exists a disturbance function $d_{2}(\cdot) \in \mathfrak{D}_2$ such that the resulting solution trajectory $x_{2}(\cdot)$ of $\Sigma_2$ (with $x_2(0)=x_{2}$) satisfies
		\begin{equation}
		\label{23}
		(x_{1}(t),x_{2}(t))\in \mathcal{R}, \, t\geq 0,
		\end{equation}
		and conversely for every disturbance function $d_{2}(\cdot)$ for which there exists a solution $x_{2}(\cdot)$ of $\Sigma_2$ (with $x_2(0)=x_{2}$), there exists a disturbance function $d_{1}(\cdot)$ such that the resulting solution trajectory $x_{1}(\cdot)$ of $\Sigma_1$ (with $x_1(0)=x_{1}$) satisfies (\ref{23}).
		\item 
		\begin{equation}
		\label{24}
		C_{1}x_{1}=C_{2}x_{2}, \mbox{ for all } (x_1,x_2) \in \mathcal{R}.
		\end{equation}
	\end{enumerate}
\end{definition}

Using the geometric notion of a {\it controlled invariant subspace} \cite{Won74,Bas92}, a linear-algebraic characterization of a bisimulation relation $\mathcal{R}$ is given in the following proposition and subsequent theorem.
\begin{proposition}\label{prop1}
	Consider two systems $\Sigma_{i}$ as in (\ref{22}), with consistent subsets $\mathcal{V}_{i}^{*},i=1,2$. A subspace $\mathcal{R}\subset\mathcal{X}_{1}\times\mathcal{X}_{2}$ satisfying $\pi_{i}(\mathcal{R})\subset \mathcal{V}_{i}^{*},i=1,2,$ is a bisimulation relation between $\Sigma_{1}$ and $\Sigma_{2}$ if and only if for all $(x_{1},x_{2})\in\mathcal{R}$ and for all $u\in\mathcal{U}$ the following properties hold :
	\begin{enumerate}
		\item For every $d_{1}\in\mathcal{D}_{1}$ for which there exists $f_{1}\in\mathcal{V}_{1}^{*}$ such that $E_{1}f_{1}=A_{1}x_{1}+B_{1}u+G_{1}d_{1}$, there exists $d_{2}\in\mathcal{D}_{2}$ for which there exists $f_{2}\in\mathcal{V}_{2}^{*}$ such that $E_{2}f_{2}=A_{2}x_{2}+B_{2}u+G_{2}d_{2}$ while
		\begin{equation}
		\label{25}
		(f_{1},f_{2})\in\mathcal{R},
		\end{equation}
		and conversely for every $d_{2}\in\mathcal{D}_{2}$ for which there exists $f_{2}\in\mathcal{V}_{2}^{*}$ such that $E_{2}f_{2}=A_{2}x_{2}+B_{2}u+G_{2}d_{2}$, there exists $d_{1}\in\mathcal{D}_{1}$ for which there exists $f_{1}\in\mathcal{V}_{1}^{*}$ such that $E_{1}f_{1}=A_{1}x_{1}+B_{1}u+G_{1}d_{1}$ while (\ref{25}) holds.
		\item \begin{equation}
		\label{26}
		C_{1}x_{1}=C_{2}x_{2}.
		\end{equation}
	\end{enumerate}
\end{proposition} 

\begin{proof} Properties (2) of Definition \ref{def1} and Proposition \ref{prop1}, cf. (\ref{24}) and (\ref{26}), are equal, so we only need to prove equivalence of Properties (1) of Definition \ref{def1} and Proposition \ref{prop1}. 
	
	In order to do this we will utilize the fact (as explained above) that the DAEs $E_i\dot{x}_i  = A_ix_i+B_iu_i+G_id_i, i=1,2,$ can be transformed, see \eqref{6}, to DAEs of the form $E_i\dot{x}_i  = A_ix_i+B_iu_i, i=1,2$, not containing disturbances. Hence it is sufficient to prove equivalence of Properties (1) of Definition \ref{def1} and Proposition \ref{prop1} for systems $\Sigma_1$ and $\Sigma_2$ of the form \eqref{7}. For clarity we will restate Property (1) in this simplified case briefly as follows:\\
	Property (1) of Definition \ref{def1}: For every solution $x_1(\cdot)$ of $\Sigma_1$ with $x_1(0)=x_1$ there exists a solution $x_2(\cdot)$ of $\Sigma_2$ with $x_2(0)=x_2$ such that \eqref{23} holds, and conversely.\\
	Property (1) of Proposition \ref{prop1}: For every $f_1 \in \mathcal{V}_1^*$ such that $E_1f_1=A_1x_1 + B_1u$ there exists $f_2 \in \mathcal{V}_2^*$ such that $E_2f_2=A_2x_2 + B_2u$ such that \eqref{25} holds, and conversely.
	
	'{\it Only if part}'.
	Take $u(\cdot)\in\mathfrak{U}$ and $(x_{1}, x_{2})\in\mathcal{R}$, and let $f_{1}\in\mathcal{V}_{1}^{*}$ be such that $E_{1}f_{1}=A_{1}x_{1}+ B_{1}u(0)$. According to Theorem \ref{thm5}, there exists a solution $x(\cdot)$ of $\Sigma_{1}$ such that $x_{1}(0)=x_{1}$ and $\dot{x}_{1}(0)=f_{1}$. Then, based on Property (1) of Definition \ref{def1}, there exists a solution $x_{2}(\cdot)$ of $\Sigma_{2}$ with $x_{2}(0)=x_{2}$ such that (\ref{23}) holds. By differentiating $x_{2}(t)$ with respect to $t$ and denoting $f_{2}:=\dot{x}_{2}(0)$, we obtain (\ref{25}). The same argument holds for the case where the indices $1$ and $2$ are interchanged.

	'{\it If part}'.
	Let $(x_1,x_2) \in \mathcal{R}$, $u(\cdot) \in \mathfrak{U}$. Consider any solution $x_{1}(\cdot)$ of $\Sigma_{1}$ corresponding to $x_{1}(0)=x_{1}$. Transform systems $\Sigma_1$ and $\Sigma_2$ into the form \eqref{9}. This means that $x_1(\cdot) = \begin{bmatrix} x^a_1(\cdot) \\ x^b_1(\cdot) \end{bmatrix}, t\geq 0,$ is a solution to
	\begin{equation}
	\label{27}
	\Sigma_{1}:\begin{array}{lllll}
	\dot{x}_{1}^{a}(t)= (A_{1}^{aa}+A_{1}^{ab}F_{1})x_{1}^{a}(t)+A_{1}^{ab}z_{1}(t)+B_{1}^{a}u(t), \; x_1^a(t) \in  \mathcal{W}_{1}\\ [2mm]
	x_{1}^{b}(t)=F_{1}x_{1}^{a}(t)+z_{1}(t), \quad z_{1}(t)\in\ker A^{bb}_{1}\cap (A_{1}^{ab})^{-1}\mathcal{W}_{1}, \; t\geq 0
	\end{array}
	\end{equation}
	Equivalently, $x^a_1(\cdot), t\geq 0,$ is a solution to
	\begin{equation}\label{28}
	\begin{array}{lllll}
	\dot{x}_{1}^{a}(t)= (A_{1}^{aa}+A_{1}^{ab}F_{1})x_{1}^{a}(t)+A_{1}^{ab}z_{1}(t)+B_{1}^{a}u(t), \; x_1^a(t) \in  \mathcal{W}_{1}\\ [2mm]
	\dot{z} _{1}(t) = e_1(t), \quad z_1(t) \in\ker A^{bb}_{1}\cap (A_{1}^{ab})^{-1}\mathcal{W}_{1} ,
	\end{array}
	\end{equation}
	where $e_1(\cdot)$ is a disturbance function, while additionally $x_{1}^{b}(t)=F_{1}x_{1}^{a}(t)+z_{1}(t), \, t \geq 0$.
	
	Similarly, the solutions $x_2(\cdot) = \begin{bmatrix} x^a_2(\cdot) \\ x^b_2(\cdot) \end{bmatrix}, t\geq 0,$ of $\Sigma_2$ are generated as solutions $x_2^a(\cdot)$ of
	\begin{equation}\label{29}
	\begin{array}{lllll}
	\dot{x}_{2}^{a}(t)= (A_{2}^{aa}+A_{1}^{ab}F_{2})x_{1}^{a}(t)+A_{2}^{ab}z_{1}(t)+B_{2}^{a}u(t), \; x_2^a(t) \in  \mathcal{W}_{2}\\ [2mm]
	\dot{z} _{2}(t) = e_2(t), \quad z_2(t) \in\ker A^{bb}_{2}\cap (A_{2}^{ab})^{-1}\mathcal{W}_{2} ,
	\end{array}
	\end{equation}
	where $e_2(\cdot)$ is a disturbance function, while additionally $x_{2}^{b}(t)=F_{2}x_{2}^{a}(t)+z_{2}(t), \, t \geq 0$.
	
	Now, the systems \eqref{28} and \eqref{29} with state vectors $\begin{bmatrix}x_1^a(t) \\z_1(t)\end{bmatrix}$, respectively $\begin{bmatrix}x_2^a(t) \\z_2(t)\end{bmatrix}$ are {\it ordinary} (no algebraic constraints) linear systems with disturbances $e_1$ and $e_2$, to which the bisimulation theory of \cite{Arjan04a} for ordinary linear systems applies. In particular, given the solution $x_1^a(\cdot), z_1(\cdot)$, and corresponding 'disturbance' $e_1(\cdot)$ by Proposition 2.9 in \cite{Arjan04a}, Property (1) in Proposition \ref{prop1} implies that there exists a disturbance $e_2(\cdot)$ with $e_2(t)=e_2(x_1^a(t), z_1(t),$ $x_2^a(t),z_2(t),e_1(t))$ such that the combined dynamics of $(x_1^a,z_1)$ and $(x_2^a,z_2)$ remain in $\mathcal{R}$. This implies Property (1) in Definition \ref{def1}.
	
The same argument holds for the case where the indices $1$ and $2$ are interchanged.
\end{proof}

The next step in the linear-algebraic characterization of bisimulation relations for linear DAE systems is provided in the following theorem.
	
	\begin{theorem}\label{thm1}
		A subspace $\mathcal{R}\subset\mathcal{X}_{1}\times\mathcal{X}_{2}$ is a bisimulation relation between $\Sigma_{1}$ and $\Sigma_{2}$ satisfying $\pi_{i}(\mathcal{R})\subset\mathcal{V}_{i}^{*},i=1,2,$ if and only if
		\begin{equation}
		\label{30}
		\begin{array}{rlllll}
		(a) & \mathcal{R}+\left[
		\begin{array}{cccc}
		E_{1}^{-1}(\Ima G_{1})\cap\mathcal{V}_{1}^{*}\\
		0
		\end{array}
		\right]=\mathcal{R}+\left[
		\begin{array}{cccc}
		0\\
		E_{2}^{-1}(\Ima G_{2})\cap\mathcal{V}_{2}^{*}
		\end{array}
		\right],\\[5mm]
		(b) & \left[
		\begin{array}{cccc}
		A_{1} & 0\\
		0 & A_{2}
		\end{array}
		\right]\mathcal{R}\subset\left[
		\begin{array}[rcl]{cccc}
		E_{1} & 0\\
		0 & E_{2}
		\end{array}
		\right]\mathcal{R}+\Ima\left[
		\begin{array}{cccc}
		G_{1} & 0\\0 & G_{2}
		\end{array}
		\right],\\[5mm]
		(c) & \Ima \left[
		\begin{array}{cccc}
		B_{1}\\
		B_{2}
		\end{array}
		\right]\subset\left[
		\begin{array}[rcl]{cccc}
		E_{1} & 0\\
		0 & E_{2}
		\end{array}
		\right]\mathcal{R}+\Ima\left[
		\begin{array}{cccc}
		G_{1} & 0\\0 & G_{2}
		\end{array}
		\right],\\[5mm]
		(d) & \mathcal{R} \subset \ker \left[
		C_{1}\vdots -C_{2}
		\right].
		\end{array}			
		\end{equation}
	\end{theorem}
	\begin{proof}
		\textit{'If part'}. Condition (\ref{26}) of Proposition \ref{prop1} follows trivially from condition (\ref{30}d). 
		From (\ref{30}b,c) it follows that for every $(x_{1},x_{2})\in\mathcal{R}$ and $u\in\mathcal{U}$ there exist $(f_{1},f_{2})\in\mathcal{R}$, and $d_{1}\in\mathcal{D}_{1}$, $d_{2}\in\mathcal{D}_{2}$, such that
		\begin{equation}
		\label{31}
		\begin{array}{rlllll}
		\left[
		\begin{array}{cccc}
		E_{1} & 0\\
		0 & E_{2}
		\end{array}
		\right]\left[
		\begin{array}{cccc}
		f_{1}\\
		f_{2}
		\end{array}
		\right]&=&\left[
		\begin{array}{cccc}
		A_{1} & 0\\
		0 & A_{2}
		\end{array}
		\right]\left[
		\begin{array}{cccc}
		x_{1}\\
		x_{2}
		\end{array}
		\right]+\left[
		\begin{array}{cccc}
		B_{1}\\
		B_{2}
		\end{array}
		\right]u+\left[
		\begin{array}{cccc}
		G_{1}\\
		0
		\end{array}
		\right]d_{1}\\[5mm]
		&+&\left[
		\begin{array}{cccc}
		0\\
		G_{2}
		\end{array}
		\right]d_{2}.
		\end{array}
		\end{equation}
		This implies $\pi_{i}(\mathcal{R}) \subset \mathcal{V}^{*}_{i}, i=1,2$. 
		
		Now let $(x_{1},x_{2})\in\mathcal{R}$ and $u\in\mathcal{U}$. Then as above, by (\ref{30}b,c), there exist $(f_{1},f_{2})\in\mathcal{R}$, and $d_{1}\in\mathcal{D}_{1}$, $d_{2}\in\mathcal{D}_{2}$ such that \eqref{31} holds. Now consider any $f_1' \in \mathcal{V}_1^*$ and $d_1' \in \mathcal{D}_1$ such that $E_{1}f'_{1}=A_{1}x_{1}+B_{1}u+G_{1}d'_{1}$. Then $f_1' = f_1 + v_1$ for some $v_1 \in E_{1}^{-1}(\Ima G_{1})\cap\mathcal{V}_{1}^{*}$. Hence by (\ref{30}a) there exists $v_2 \in E_{2}^{-1}(\Ima G_{2})\cap\mathcal{V}_{2}^{*}$ and $(f''_{1},f''_{2})\in\mathcal{R}$ such that
		\[
		\begin{bmatrix} v_1 \\ 0 \end{bmatrix} = \begin{bmatrix} f''_{1} \\f''_{2} \end{bmatrix} - \begin{bmatrix} 0 \\v_2 \end{bmatrix}
		\]
		with $E_2v_2=G_2d''_2$ for some $d''_2 \in \mathcal{D}_2$. Therefore
		\[
		\begin{bmatrix} f'_1 \\ f_2 \end{bmatrix}  =  \begin{bmatrix} f_1 \\ f_2 \end{bmatrix} + \begin{bmatrix} v_1 \\ 0 \end{bmatrix} 
		= \begin{bmatrix} f_1 \\ f_2 \end{bmatrix} +  \begin{bmatrix} f''_1 \\ f''_2 \end{bmatrix} -  \begin{bmatrix} 0 \\v_2 \end{bmatrix} 
		=  \begin{bmatrix} f'_1 \\ f'_2 \end{bmatrix}  -  \begin{bmatrix} 0 \\ v_2 \end{bmatrix},
		\]
		with $f_2' := f_2 + f''_2$. Clearly $(f'_1, f'_2) \in \mathcal{R}$. It follows that
		\[
		E_2 f'_2 = E_2f_2 + E_2v_2 = A_{2}x_{2}+B_{2}u+ G_{2}d'_{2},
		\]
		with $d'_2:=d_2 + d''_2$. Similarly, for every $f'_2 \in \mathcal{V}_2^*$ and $d_2' \in \mathcal{D}_2$ such that $E_{2}f'_{2}=A_{2}x_{2}+B_{2}u+G_{2}d'_{2}$ there exist $f'_1 \in \mathcal{V}_1^*$ with $(f'_1, f'_2) \in \mathcal{R}$, while $E_1 f'_1 = A_{1}x_{1}+B_{1}u+ G_{1}d'_{1}$ for some $d'_1:=d_1 + d''_1$. Hence we have shown property (1) of Proposition \ref{prop1}.

		\textit{'Only if part'}. Property (2) of Proposition \ref{prop1} is trivially equivalent with (\ref{30}d). Since $\pi_{i}(\mathcal{R})\subset\mathcal{V}_{i}^{*}$ for $i=1,2$ we have
		\begin{equation}
		\label{32}
		\left[
		\begin{array}{cccc}
		A_{1} & 0\\
		0 & A_{2}
		\end{array}
		\right]\mathcal{R}\subset\left[
		\begin{array}{cccc}
		E_{1} & 0\\
		0 & E_{2}
		\end{array}
		\right]\mathcal{R}+\Ima\left[
		\begin{array}{cccc}
		G_{1}&0\\ 0 & G_{2}
		\end{array}
		\right]
		\end{equation} 
		and
		\begin{equation}
		\label{33}
		\begin{array}{rlllll}
		\Ima\left[
		\begin{array}{cccc}
		B_{1}\\
		B_{2}
		\end{array}
		\right]&\subset&\left[
		\begin{array}{cccc}
		E_{1} & 0\\
		0 & E_{2}
		\end{array}
		\right]\mathcal{R}+\Ima\left[
		\begin{array}{cccc}
		G_{1}&0\\0 & G_{2}
		\end{array}
		\right].
		\end{array}
		\end{equation}
		Furthermore, since property (1) of Proposition \ref{prop1} holds, by taking $(x_1,x_2)=(0,0)$ and $u=0$, then for every $d_{1}$ for which there exists $f_{1}\in\mathcal{V}_{1}^{*}$ such that $E_{1}f_{1}=G_{1}d_{1}$, there exists $d_{2}$ and $f_{2}\in\mathcal{V}_{2}^{*}$ such that $E_{2}f_{2}=G_{2}d_{2}$, while $(f_{1},f_{2})\in\mathcal{R}$. Hence
		\begin{equation}
		\label{34}
		\left[
		\begin{array}{cccc}
		f_{1}\\0
		\end{array}
		\right]=\left[
		\begin{array}{cccc}
		f_{1}\\f_{2}
		\end{array}
		\right]-\left[
		\begin{array}{cccc}
		0\\f_{2}
		\end{array}
		\right]\in\mathcal{R}+\left[
		\begin{array}{cccc}
		0\\E_{2}^{-1}(\Ima G_{2})\cap\mathcal{V}_{2}^{*}
		\end{array}
		\right],
		\end{equation}
		and thus
		\begin{equation}
		\label{35}
		\left[
		\begin{array}{cccc}
		E_{1}^{-1}(\Ima G_{1})\cap\mathcal{V}_{1}^{*}\\
		0
		\end{array}
		\right]\subset\mathcal{R}+\left[
		\begin{array}{cccc}
		0\\
		E_{2}^{-1}(\Ima G_{2})\cap \mathcal{V}_{2}^{*}
		\end{array}
		\right].
		\end{equation} 
		Similarly one obtains 
		\begin{equation}
		\label{36}
		\left[
		\begin{array}{cccc}
		0\\
		E_{2}^{-1}(\Ima G_{2})\cap\mathcal{V}_{2}^{*}
		\end{array}
		\right]\subset\mathcal{R}+\left[
		\begin{array}{cccc}
		E_{1}^{-1}(\Ima G_{1})\cap\mathcal{V}_{1}^{*}\\
		0	
		\end{array}
		\right]
		\end{equation}
		Combining equations (\ref{35}) and (\ref{36}) implies condition (\ref{30}a). 
	\end{proof}
	\begin{remark}
		In the special case $E_{i}, i=1,2,$ equal to the identity matrix, it follows that $\mathcal{V}_i^* = \mathcal{X}_i, i=1,2$, and \eqref{30} reduces to
		\begin{equation}
		\label{37}
		\begin{array}{rlllll}
		(a) & \mathcal{R}+\left[
		\begin{array}{cccc}
		\Ima G_{1}\\
		0
		\end{array}
		\right]
		=\mathcal{R}+
		\left[
		\begin{array}{cccc}
		0\\
		\Ima G_{2}
		\end{array}
		\right] =: \mathcal{R}_e,
		\\[5mm]
		(b) & \left[
		\begin{array}{cccc}
		A_{1} & 0\\
		0 & A_{2}
		\end{array}
		\right]\mathcal{R}\subset
		\mathcal{R}+\Ima\left[
		\begin{array}{cccc}
		G_{1} & 0\\0 & G_{2} 
		\end{array}
		\right],\\[5mm]
		(c) & \Ima \left[
		\begin{array}{cccc}
		B_{1}\\
		B_{2}
		\end{array}
		\right]\subset
		\mathcal{R}+\Ima\left[
		\begin{array}{cccc}
		G_{1} & 0\\0 & G_{2}
		\end{array}
		\right],\\[5mm]
		(d) & \mathcal{R} \subset \ker \left[
		C_{1}\vdots -C_{2}
		\right].
		\end{array}			
		\end{equation}
		Hence in this case Theorem \ref{thm1} reduces to \cite[Theorem 2.10]{Arjan04a}.
	\end{remark}
	
%%%%%%%%%%%%%%%%%%%%%%%%%%%%%%%%%%%%%%%%%%%%%%%%%%%%%%%%%%%%%%%%%%%%%%%%%%%%%%%%	
	\subsection{Computing the maximal bisimulation relation}
	The {\it maximal} bisimulation relation between two DAE systems, denoted $\mathcal{R}^{max}$, can be computed, whenever it exists, in the following way, similarly to the well-known algorithm \cite{Won74,Bas92} from geometric control theory to compute the {\it maximal controlled invariant subspace}. For notational convenience define
	\begin{equation}
	\label{38}
	\begin{array}{lllllll}
	E^{\times}:=\left[
	\begin{array}{cccc}
	E_{1} & 0\\
	0 & E_{2}
	\end{array}
	\right], & A^{\times}:=\left[
	\begin{array}{cccc}
	A_{1} & 0\\
	0 & A_{2}
	\end{array}
	\right],& C^{\times}:=
	\left[ C_{1}\vdots -C_{2} \right]
	,\\[5mm]
	\mathcal{G}_{1}^{\times}:=\left[
	\begin{array}{ccccc}
	E_{1}^{-1}(\Ima G_{1})\cap\mathcal{V}_{1}^{*}\\
	0
	\end{array}
	\right], &
	\mathcal{G}_{2}^{\times}:=\left[
	\begin{array}{ccc}
	0\\
	E_{2}^{-1}(\Ima G_{2})\cap \mathcal{V}_{2}^{*}
	\end{array}
	\right], & \bar{G}^{\times}:=\left[
	\begin{array}{cccc}
	G_{1} & 0\\
	0 & G_{2}
	\end{array}
	\right].
	\end{array}
	\end{equation}
	\begin{algorithm}\label{alg1}
		Given two systems $\Sigma_{1}$ and $\Sigma_{2}$. Define the following sequence $\mathcal{R}^{j}, j=0,1,2,...$, of subsets of $\mathcal{X}_{1}\times\mathcal{X}_{2}$
		\begin{equation}
		\label{39}
		\begin{array}{rlllllll}
		\mathcal{R}^{0} &= \mathcal{X}_{1}\times\mathcal{X}_{2},\\
		\mathcal{R}^{1} &= \{z\in\mathcal{R}^{0} \mid z\in\ker C^{\times},\mathcal{R}^{1}+\mathcal{G}_{1}^{\times}=\mathcal{R}^{1}+\mathcal{G}_{2}^{\times}\},\\[5mm]
		\mathcal{R}^{2} &= \{z\in\mathcal{R}^{1} \mid A^{\times}z\subset E^{\times}\mathcal{R}^{1}+\Ima \bar{G}^{\times},\mathcal{R}^{2}+\mathcal{G}_{1}^{\times}=\mathcal{R}^{2}+\mathcal{G}_{2}^{\times}\},\\
		&\vdots\\
		\mathcal{R}^{j} &= \{z\in\mathcal{R}^{j-1} \mid A^{\times}z+\subset E^{\times}\mathcal{R}^{j-1}+\Ima \bar{G}^{\times},\mathcal{R}^{j}+\mathcal{G}_{1}^{\times}=\mathcal{R}^{j}+\mathcal{G}_{2}^{\times} \}.
		\end{array}
		\end{equation}
	\end{algorithm}
	\begin{proposition}\label{prop2}
		The sequence $\mathcal{R}^{0}, \mathcal{R}^{1},...,\mathcal{R}^{j},...$ satisfies the following properties.
		\begin{enumerate}
			\item $\mathcal{R}^{j}, j\neq 0$, is a linear space or empty. Furthermore $\mathcal{R}^{0}\supset\mathcal{R}^{1}\supset\mathcal{R}^{2}\supset\cdots\supset\mathcal{R}^{j}\supset\mathcal{R}^{j+1}\supset\cdots$.
			\item There exists a finite $k$ such that $\mathcal{R}^{k}=\mathcal{R}^{k+1}=:\mathcal{R}^{*}$, and then $\mathcal{R}^{j}=\mathcal{R}^{*}$ for all $j\neq k$.
			\item $\mathcal{R}^{*}$ is either empty or equals the maximal subspace of $\mathcal{X}_{1}\times\mathcal{X}_{2}$ satisfying the properties
			\begin{equation}
			\label{40}
			\begin{array}{rlllll}
			(i) & \mathcal{R}^{*}+\left[
			\begin{array}{cccc}
			E_{1}^{-1}(\Ima G_{1})\cap\mathcal{V}_{1}^{*}\\
			0
			\end{array}
			\right]=\mathcal{R}^{*}+\left[
			\begin{array}{cccc}
			0\\
			E_{2}^{-1}(\Ima G_{2})\cap\mathcal{V}_{2}^{*}
			\end{array}
			\right],\\[5mm]
			(ii) & \left[
			\begin{array}{cccc}
			A_{1} & 0\\
			0 & A_{2}
			\end{array}
			\right]\mathcal{R}^{*}\subset\left[
			\begin{array}[rcl]{cccc}
			E_{1} & 0\\
			0 & E_{2}
			\end{array}
			\right]\mathcal{R}^{*}+\Ima\left[
			\begin{array}{cccc}
			G_{1} & 0\\ 0 & G_{2}
			\end{array}
			\right],\\[5mm]
			(iii) & \mathcal{R}^{*} \subset 
			\ker\left[
			C_{1}\vdots -C_{2}
			\right].
			\end{array}			
			\end{equation}
		\end{enumerate}
	\end{proposition} 
	\begin{proof}
		Analogous to the proof of \cite[Theorem 3.4]{Arjan04a}.
	\end{proof}
	If $\mathcal{R}^{*}$ as obtained from Algorithm \ref{alg1} is non-empty and satisfies condition (\ref{30}c) in Theorem \ref{thm1}, then it follows that $\mathcal{R}^{*}$ is the \textit{maximal bisimulation relation} $\mathcal{R}^{max}$ between $\Sigma_{1}$ and $\Sigma_{2}$, while if $\mathcal{R}^{*}$ is empty or does not satisfy condition (\ref{30}c) in Theorem \ref{thm1} then there does not exist any bisimulation relation between $\Sigma_{1}$ and $\Sigma_{2}$. 
	
	Furthermore two systems are called {\it bisimilar} if there exists a bisimulation relation relating {\it all} states. This is formalized in the following definition and corollary.	
	
	\begin{definition}\label{def2}
		Two systems $\Sigma_{1}$ and $\Sigma_{2}$ as in (\ref{22}) are \textit{bisimilar}, denoted $\Sigma_{1}\sim\Sigma_{2}$, if there exists a bisimulation relation $\mathcal{R}\subset\mathcal{X}_{1}\times\mathcal{X}_{2}$ with the property that 
		\begin{equation}
		\label{41}
		\begin{array}{ccccc}
		\pi_{1}(\mathcal{R})=\mathcal{V}_{1}^{*}, & \pi_{2}(\mathcal{R})=\mathcal{V}_{2}^{*},
		\end{array}
		\end{equation}
		where $\mathcal{V}_{i}^{*}$ is the consistent subset of $\Sigma_{i},i=1,2$.
	\end{definition}
	
	\begin{corollary}\label{cor1}
		$\Sigma_{1}$ and $\Sigma_{2}$ are bisimilar if and only if $\mathcal{R}^{*}$ is non-empty and satisfies condition (\ref{30}c) in Theorem \ref{thm1} and equation (\ref{41}).
	\end{corollary}
	Bisimilarity is implying the {\it equality of external behavior}. Consider two systems $\Sigma_{i}, i=1,2,$ as in (\ref{22}), with external behavior $\mathcal{B}_{i}$ defined as
	\begin{equation}
	\label{42}
	\mathcal{B}_{i}:=\{(u_{i}(\cdot), y_{i}(\cdot))\mid\exists x_{i}(\cdot), d_{i}(\cdot) \text{ such that } (\ref{22}) \text{ is satisfied}\}.
	\end{equation}
	Analogously to \cite{Arjan04a} we have the following result.
	\begin{proposition}\label{prop3}
		Let $\Sigma_{i}, i=1,2,$ be bisimilar. Then their external behaviors $\mathcal{B}_{i}$ are equal.
	\end{proposition}
	
	However, due to the possible non-determinism introduced by the matrices $G$ and $E$ in (\ref{1}), two systems of the form (\ref{1}) may have the same external behavior while not being bisimilar. This is already illustrated in \cite{Arjan04a} for the case $E=I$.

%%%%%%%%%%%%%%%%%%%%%%%%%%%%%%%%%%%%%%%%%%%%%%%%%%%%%%%%%%%%%%%%%%%%%%%%%%%%%%%%	
	\subsection{Bisimulation relation for deterministic case}
	In this section, we specialize the results to DAE systems {\it without} disturbances $d$.
	Consider two systems of the form
	\begin{equation}
	\label{43}
	\Sigma_{i}:\begin{array}{rll}
	E_{i}\dot{x}_{i} =& A_{i}x_{i}+B_{i}u_{i}, & \quad x_{i}\in\mathcal{X}_{i}, u_{i}\in\mathcal{U},\\[2mm]
	y_{i}=& C_{i}x_{i}, & \quad y_{i}\in\mathcal{Y}, i=1,2,
	\end{array}
	\end{equation}
	where $E_{i},A_{i}\in \mathbb{R}^{q_i\times n_i}$ and $B_{i}\in\mathbb{R}^{q_i\times m}, C_{i} \in \mathbb{R}^{ p \times n_i}$ for $i=1,2$. Theorem \ref{thm1} can be specialized as follows.
	\begin{corollary}\label{cor2}
		A subspace $\mathcal{R}\subset\mathcal{X}_{1}\times\mathcal{X}_{2}$ is a bisimulation relation between $\Sigma_{1}$ and $\Sigma_{2}$ given by (\ref{43}), satisfying $\pi_{i}(\mathcal{R})\subset\mathcal{V}_{i}^{*},i=1,2,$ if and only if
		\begin{equation}
		\label{44}
		\begin{array}{rllll}
		(a) &\mathcal{R}+\left[
		\begin{array}{cccc}
		\ker E_{1}\cap\mathcal{V}_{1}^{*}\\
		0
		\end{array}
		\right]=\mathcal{R}+\left[
		\begin{array}{cccc}
		0\\
		\ker E_{2}\cap\mathcal{V}_{2}^{*}
		\end{array}
		\right],\\[5mm]
		(b) & \left[
		\begin{array}{cccc}
		A_{1} & 0\\
		0 & A_{2}
		\end{array}
		\right]\mathcal{R}\subset\left[
		\begin{array}{cccc}
		E_{1} & 0\\
		0 & E_{2}
		\end{array}
		\right]\mathcal{R},\\[5mm]
		(c) & \Ima \left[
		\begin{array}{cccc}
		B_{1}\\
		B_{2}
		\end{array}
		\right]\subset\left[
		\begin{array}{cccc}
		E_{1} & 0\\
		0 & E_{2}
		\end{array}
		\right]\mathcal{R},\\[5mm]
		(d) & \mathcal{R} \subset \ker\left[
		C_{1}\vdots -C_{2}
		\right]	.
		\end{array}	
		\end{equation}
	\end{corollary}
	Corollary \ref{cor2} can be applied to the following situation considered in \cite{Arjan04a}. Consider two linear systems given by
	\begin{equation}
	\label{45}
	\Sigma_{i}:\begin{array}{rcll}
	\dot{x}_{i} &=& A_{i}x_{i}+B_{i}u_{i}+G_{i}d_{i}, \\[2mm]
	y_{i} &=& C_{i}x_{i}.
	\end{array}
	\end{equation}
	By multiplying both sides of the first equation of (\ref{45}) by an annihilating matrix $G_{i}^{\perp}$ of maximal rank one obtains the equivalent system representation without disturbances
	\begin{equation}
	\label{46}
	\begin{array}{rlllll}
	G_{i}^{\perp}\dot{x}_{i}&=& G_{i}^{\perp}A_{i}x_{i}+G_{i}^{\perp}B_{i}u_{i},\\[2mm]
	y_{i} &=& C_{i}x_{i},
	\end{array}
	\end{equation}
	which is of the general form (\ref{43}); however satisfying the special property $\mathcal{V}_i^* = \mathcal{X}_{i}$. This implies that $\mathcal{R}$ is a bisimulation relation between $\Sigma_1$ and $\Sigma_2$ given by (\ref{45}) {\it if and only if} it is a bisimulation relation between $\Sigma_1$ and $\Sigma_2$ given by (\ref{46}), as can be seen as follows. 
	As already noted in Remark 2.6 a bisimulation relation between $\Sigma_{1}$ and $\Sigma_{2}$ as in (\ref{45}) is a subspace  $\mathcal{R}\subset\mathcal{X}_{1}\times\mathcal{X}_{2}$ satisfying \eqref{37}.
	Now let $\mathcal{R}$ satisfy (\ref{37}). We will show that it will satisfy (\ref{44}) for systems (\ref{46}). First, since $\mathcal{V}_i = \mathcal{X}_{i}$ and $\ker E_i = \ker G_i^{\perp} = \Ima G_i$ we see that (\ref{44}a) is satisfied. Furthermore, by pre-multiplying both sides of (\ref{37}b,c) with
	\begin{equation}
	\label{47}
	\left[
	\begin{array}{cccc}
	G_{1}^{\perp} & 0\\
	0 & G_{2}^{\perp}
	\end{array}
	\right],
	\end{equation}
	we obtain
	\begin{equation}
	\label{48}
	\begin{array}{rllll}
	\left[
	\begin{array}{ccccc}
	G_{1}^{\perp}A_{1} & 0\\
	0 & G_{1}^{\perp}A_{2}
	\end{array}
	\right] \mathcal{R} &\subset&\left[
	\begin{array}{ccccc}
	G_{1}^{\perp} & 0\\
	0 & G_{2}^{\perp}
	\end{array}
	\right]\mathcal{R},\\[5mm]
	\Ima\left[
	\begin{array}{ccccc}
	G_{1}^{\perp}B_{1}\\
	G_{2}^{\perp}B_{2}
	\end{array}
	\right]&\subset&\left[
	\begin{array}{ccccc}
	G_{1}^{\perp} & 0\\
	0 & G_{2}^{\perp}
	\end{array}		
	\right]\mathcal{R},
	\end{array}
	\end{equation}
	showing satisfaction of (\ref{44}b,c). 
	Conversely, let $\mathcal{R}$ be a bisimulation relation between $\Sigma_{1}$ and $\Sigma_{2}$ given by (\ref{46}), having consistent subsets $\mathcal{V}_{i}^{*}=\mathcal{X}_{i}, i=1,2.$ Then according to (\ref{44}) it is satisfying
	\begin{equation}
	\label{49}
	\begin{array}{rlllll}
	(a) & \mathcal{R}+\left[
	\begin{array}{cccc}
	\ker G_{1}^{\perp}\\
	0
	\end{array}
	\right]=\mathcal{R}+\left[
	\begin{array}{cccc}
	0\\
	\ker G_{2}^{\perp}
	\end{array}
	\right],\\[5mm]
	(b) & \left[
	\begin{array}{ccccc}
	G_{1}^{\perp}A_{1} & 0\\
	0 & G_{1}^{\perp}A_{2}
	\end{array}
	\right]\subset\left[
	\begin{array}{ccccc}
	G_{1}^{\perp} & 0\\
	0 & G_{2}^{\perp}
	\end{array}
	\right]\mathcal{R},\\[5mm]
	(c) & \Ima\left[
	\begin{array}{ccccc}
	G_{1}^{\perp}B_{1}\\
	G_{2}^{\perp}B_{2}
	\end{array}
	\right]\subset\left[
	\begin{array}{ccccc}
	G_{1}^{\perp} & 0\\
	0 & G_{2}^{\perp}
	\end{array}		
	\right]\mathcal{R},\\[5mm]
	(d) & \mathcal{R} \subset \ker\left[
	C_{1} \vdots -C_{2}
	\right].
	\end{array}
	\end{equation}
	Using again $\Ima G_{i}=\ker G_{i}^{\perp}$ it immediately follows that $\mathcal{R}$ is satisfying (\ref{37}), and thus is a bisimulation relation between the systems (\ref{45}).

	%%%%%%%%%%%%%%%%%%%%%%%%%%%%%%%%%%%%%%%%%%%%%%%%%%%%%%%%%%%%%%%%%%%%%%%%%%%%%%%%
	\section{Bisimulation relations for regular DAE systems}
	In this section we will specialize the notion of bisimulation relation for general DAE systems of the form (\ref{1}) to {\it regular} DAE systems. Regularity is usually defined for  DAE systems {\it without disturbances}
	\begin{equation}
	\label{50}
	\Sigma:\begin{array}{rllll}
	E\dot{x}&=& Ax +Bu, & x\in\mathcal{X}, u\in\mathcal{U}\\[2mm]
	y &=& Cx, & y\in\mathcal{Y},
	\end{array}
	\end{equation}
	Hence the consistent subset $\mathcal{V}^*$ is either empty or equal to the maximal subspace $\mathcal{V} \subset \mathcal{X}$ satisfying $A\mathcal{V} \subset E\mathcal{V}, \Ima B \subset E\mathcal{V}$. 
	\begin{definition}
		The matrix pencil $sE-A$ is called {\it regular} if the polynomial $\det(sE-A)$ in $s \in \mathbb{C}$ is {\it not} identically zero.
		The corresponding DAE system (\ref{50}) is called regular whenever the pencil $sE-A$ is regular.  
	\end{definition}
	Define additionally $\mathcal{V}_0^*$ as the maximal subspace $\mathcal{V} \subset \mathcal{X}$ satisfying $A\mathcal{V} \subset E\mathcal{V}$. (Note that {\it if} there exists a subspace $\mathcal{V}$ satisfying $A\mathcal{V} \subset E\mathcal{V}, \Ima B \subset E\mathcal{V}$ then $\mathcal{V}_0^* =\mathcal{V}^*$.) Then \cite{Arm84}
	\begin{theorem}
		Consider (\ref{50}). The following statements are equivalent :
		\begin{enumerate}
			\item $sE-A$ is a regular pencil,
			\item $\mathcal{V}_0^{*}\cap \ker E=0$.
		\end{enumerate}
	\end{theorem}
	Regularity thus means uniqueness of solutions from any initial condition in the consistent subset $\mathcal{V}^*$ of (\ref{50}). We immediately obtain the following consequence of Corollary \ref{cor2}.
	\begin{corollary}\label{cor4}
		A subspace $\mathcal{R}\subset\mathcal{X}_{1}\times\mathcal{X}_{2}$ is a bisimulation relation between $\Sigma_{1}$ and $\Sigma_{2}$ satisfying $\pi_{i}(\mathcal{R})\subset\mathcal{V}_{i}^{*}, i=1,2$, if and only if
		\begin{equation}
		\label{51}
		\begin{array}{rllllllll}
		(a) & \left[
		\begin{array}{cccc}
		A_{1} & 0\\
		0 & A_{2}
		\end{array}
		\right]\mathcal{R}\subset\left[
		\begin{array}{cccc}
		E_{1} & 0\\
		0 & E_{2}
		\end{array}
		\right]\mathcal{R},\\[5mm]
		(b) & \Ima \left[
		\begin{array}{cccc}
		B_{1}\\
		B_{2}
		\end{array}
		\right]\subset\left[
		\begin{array}{cccc}
		E_{1} & 0\\
		0 & E_{2}
		\end{array}
		\right]\mathcal{R},\\[5mm]
		(c) & \mathcal{R}\subset\ker\left[
		C_{1}\vdots -C_{2}
		\right].
		\end{array}
		\end{equation}
	\end{corollary}
	In the regular case, the existence of a bisimulation relation can be characterized in terms of {\it transfer matrices}.
	\begin{theorem}\label{thm3}
		Let $\mathcal{R}$ be a bisimulation relation between regular systems $\Sigma_{1}$ and $\Sigma_{2}$ given in (\ref{43}), then their transfer matrices $G_{i}(s):=C_{i}(sE_{i}-A_{i})^{-1}B_{i}$ for $i=1,2$ are equal.
	\end{theorem}
	\begin{proof}
		Let $\mathcal{R}$ be a bisimulation relation between $\Sigma_{1}$ and $\Sigma_{2}$ thus it is satisfying (\ref{51}). According to (\ref{51}a) and (\ref{51}b), for $(x_{1},x_{2})\in\mathcal{R}$ and $u\in\mathcal{U}$, there exist $(\dot{x}_{1},\dot{x}_{2})\in\mathcal{R}$ such that
		\begin{equation}
		\label{52}
		\left[
		\begin{array}{cccc}
		E_{1} & 0\\
		0 & E_{2}
		\end{array}
		\right]\left[
		\begin{array}{cccc}
		\dot{x}_{1}\\
		\dot{x}_{2}
		\end{array}
		\right]=\left[
		\begin{array}{cccc}
		A_{1} & 0\\
		0 & A_{2}
		\end{array}
		\right]\left[
		\begin{array}{cccc}
		x_{1}\\
		x_{2}
		\end{array}
		\right]+\left[
		\begin{array}{cccc}
		B_{1}\\
		B_{2}
		\end{array}
		\right]u.
		\end{equation}
		Taking the \textit{Laplace} transform of (\ref{52}), we have
		\begin{equation}
		\label{53}
		\left[
		\begin{array}{cccc}
		X_{1}(s)\\
		X_{2}(s)
		\end{array}
		\right]=\left[
		\begin{array}{cccc}
		(sE_{1}-A_{1})^{-1}B_{1}\\
		(sE_{2}-A_{2})^{-1}B_{2}
		\end{array}
		\right].
		\end{equation}
		Since (\ref{51}c) holds and taking \textit{Laplace} tranform, we have
		\begin{equation}
		\label{54}
		C_{1}(sE_{1}-A_{1})^{-1}B_{1}=C_{2}(sE_{2}-A_{2})^{-1}B_{2}.
		\end{equation}
	\end{proof}
	The converse statement holds provided the matrices $E_i$ are {\it invertible}. 
	\begin{theorem}\label{thm4}
		Assume $E_i, i=1,2,$ is invertible. Then there exists a bisimulation relation $\mathcal{R}$ between $\Sigma_{1}$ and $\Sigma_{2}$ if and only if their transfer matrices $G_{i}(s):=C_{i}(sE_{i}-A_{i})^{-1}B_{i}$ for $i=1,2$ are equal.
	\end{theorem}
	\begin{proof}
		Let $G_{1}(s)=G_{2}(s)$. Then 
		\begin{equation}
		\label{55}
		\mathcal{R}:= \Ima \left[
		\begin{array}{cccccccccccc}
		E_{1}^{-1}B_{1} & E_{1}^{-1}A_{1}E_{1}^{-1}B_{1} & (E_{1}^{-1}A_{1})^{2}E_{1}^{-1}B_{1}&\cdots\\
		E_{2}^{-1}B_{2} & E_{2}^{-1}A_{1}E_{2}^{-1}B_{2} & (E_{2}^{-1}A_{2})^{2}E_{2}^{-1}B_{2}&\cdots\\
		\end{array}
		\right]
		\end{equation}
		satisfies (\ref{51}).
	\end{proof}
	
	The following example shows that Theorem \ref{thm4} does not hold if $E_{i}$ is not invertible.
	
	\begin{example}
		Consider two systems, given by
		\[
		\begin{array}{rllll}
		&\Sigma_{1} : \begin{array}{rlllll}
		\left[
		\begin{array}{cccc}
		1 & 0\\
		0 & 0
		\end{array}
		\right]\dot{x}_{1}&=&\left[
		\begin{array}{cccc}
		1 & 0\\
		0 & 1
		\end{array}
		\right]x_{1}+\left[
		\begin{array}{cccc}
		0\\
		1
		\end{array}
		\right]u_{1},\\[3mm]
		y_{1}&=&\left[
		\begin{array}{ccc}
		1 & 1
		\end{array}
		\right]x_{1},
		\end{array}\\\\
		&\Sigma_{2} :
		\begin{array}{rllll}
		\left[
		\begin{array}{cccc}
		0 & 0\\
		0 & 1
		\end{array}
		\right]\dot{x}_{2}&=&\left[
		\begin{array}{cccc}
		1 & 0\\
		0 & 1
		\end{array}
		\right]x_{2}+\left[
		\begin{array}{cccc}
		1\\
		0
		\end{array}
		\right]u_{2},\\[3mm]
		y_{2}&=&\left[
		\begin{array}{cccc}
		1 & 1
		\end{array}
		\right]x_{2}.
		\end{array}
		\end{array}\]
		System $\Sigma_{1}$ and $\Sigma_{2}$ are regular and their transfer matrices are equal. However, there does not exist any bisimulation relation $\mathcal{R}$ satisfying (\ref{51}), since in fact the consistent subsets for both system are empty.
	\end{example}
	
	%%%%%%%%%%%%%%%%%%%%%%%%%%%%%%%%%%%%%%%%%%%%%%%%%%%%%%%%%%%%%%%%%%%%%%%%%%%%%%%%
	\section{Simulation relations and abstractions}
	In this section we will define a one-sided version of the notion of bisimulation relation and bisimilarity.
	\begin{definition}
		A subspace 
		\begin{equation}
		\label{56}
		\mathcal{S}\subset\mathcal{X}_{1}\times\mathcal{X}_{2},
		\end{equation}
		with $\pi_{i}(\mathcal{S})\subset \mathcal{V}_{i}^{*}$, for i=1,2, is a \textit{simulation relation} of $\Sigma_{1}$ by $\Sigma_{2}$ with consistent subsets $\mathcal{V}_{i}^{*}, i=1,2$ if and only if for all pairs of initial conditions $(x_{1},x_{2})\in\mathcal{S}$ and any joint input function $u_{1}(\cdot)=u_{2}(\cdot)=u(\cdot)\in\mathfrak{U}
		$ the following properties hold:
		\begin{enumerate}
			\item for every disturbance function $d_{1}(\cdot) \in \mathfrak{D}_1$ for which there exists a solution $x_{1}(\cdot)$ of $\Sigma_1$ (with $x_1(0)=x_{1}$), there exists a disturbance function $d_{2}(\cdot) \in \mathfrak{D}_2$ such that the resulting solution trajectory $x_{2}(\cdot)$ of $\Sigma_2$ (with $x_2(0)=x_{2}$) satisfies for all $t\geq 0$
			\begin{equation}
			\label{57}
			(x_{1}(t),x_{2}(t))\in \mathcal{S},
			\end{equation}
			\item 
			\begin{equation}
			\label{58}
			C_{1}x_{1}=C_{2}x_{2},\quad \mbox{ for all } (x_1,x_2) \in \mathcal{S}.
			\end{equation}
			
		\end{enumerate}
		$\Sigma_{1}$ is {\it simulated} by $\Sigma_{2}$ if the simulation relation $\mathcal{S}$ satisfies $\pi_{1}(S)=\mathcal{V}_{1}^{*}$.
	\end{definition}
	
	The one-sided version of Theorem \ref{thm1} is given as follows.
	\begin{proposition}
		A subspace $\mathcal{S}\subset\mathcal{X}_{1}\times\mathcal{X}_{2}$ is a simulation relation of $\Sigma_{1}$ by $\Sigma_{2}$ satisfying $\pi_{i}(\mathcal{S})\subset\mathcal{V}_{i}^{*},$ for $i=1,2$ if and only if
		\begin{equation}
		\label{59}
		\begin{array}{rlllll}
		(a) & \mathcal{S}+\left[
		\begin{array}{cccc}
		E_{1}^{-1}(\Ima G_{1})\cap\mathcal{V}_{1}^{*}\\
		0
		\end{array}
		\right]\subset\mathcal{S}+\left[
		\begin{array}{cccc}
		0\\
		E_{2}^{-1}(\Ima G_{2})\cap\mathcal{V}_{2}^{*}
		\end{array}
		\right],\\[5mm]
		(b) & \left[
		\begin{array}{cccc}
		A_{1} & 0\\
		0 & A_{2}
		\end{array}
		\right]\mathcal{S}\subset\left[
		\begin{array}{cccc}
		E_{1} & 0\\
		0 & E_{2}
		\end{array}
		\right]\mathcal{S}+\Ima\left[
		\begin{array}{cccc}
		G_{1}& 0\\
		0 & G_{2}
		\end{array}
		\right],\\[5mm]
		(c) & \Ima \left[
		\begin{array}{cccc}
		B_{1}\\
		B_{2}
		\end{array}
		\right]\subset\left[
		\begin{array}{cccc}
		E_{1} & 0\\
		0 & E_{2}
		\end{array}
		\right]\mathcal{S}+\Ima\left[
		\begin{array}{cccc}
		G_{1}& 0\\
		0 & G_{2}
		\end{array}
		\right],\\[5mm]
		(d) & \mathcal{S} \subset \ker\left[
		C_{1}\vdots -C_{2}
		\right].
		\end{array}			
		\end{equation}
	\end{proposition}
	
	The maximal simulation relation $\mathcal{S}^{max}$ can be computed by the following simplified version of Algorithm \ref{alg1}.
	\begin{algorithm}
		Given two dynamical systems $\Sigma_{1}$ and $\Sigma_{2}$. Define the following sequence $\mathcal{S}^{j}, j=0,1,2,...$, of subsets of $\mathcal{X}_{1}\times\mathcal{X}_{2}$
		\begin{equation}
		\label{60}
		\begin{array}{rlllllll}
		\mathcal{S}^{0} &= \mathcal{X}_{1}\times\mathcal{X}_{2},\\
		\mathcal{S}^{1} &= \{z\in\mathcal{S}^{0}|z\in\ker C^{\times},\mathcal{S}^{1}+\mathcal{G}_{1}^{\times}\subset\mathcal{S}^{1}+\mathcal{G}_{2}^{\times}\}\\[5mm]
		\mathcal{S}^{2} &= \{z\in\mathcal{S}^{1}|A^{\times}z+\subset E^{\times}\mathcal{S}^{1}+\Ima \bar{G}^{\times},\mathcal{S}^{2}+\mathcal{G}_{1}^{\times}\subset\mathcal{S}^{2}+\mathcal{G}_{2}^{\times}\},\\
		&\vdots\\
		\mathcal{S}^{j} &= \{z\in\mathcal{S}^{j-1}|A^{\times}z+\subset E^{\times}\mathcal{S}^{j-1}+\Ima \bar{G}^{\times},\mathcal{S}^{j}+\mathcal{G}_{1}^{\times}\subset\mathcal{S}^{j}+\mathcal{G}_{2}^{\times}\}.
		\end{array}
		\end{equation}
	\end{algorithm}
	
	Recall the definition of the inverse relation $\mathcal{T}^{-1}:=\{(x_{a}, x_{b})\mid (x_{b},x_{a})\in\mathcal{T}\}$. We have the following
	\begin{proposition}
		Let $\mathcal{S}\subset \mathcal{X}_{1}\times\mathcal{X}_{2}$ be a simulation relation of $\Sigma_{1}$ by $\Sigma_{2}$ and let $\mathcal{T}\subset\mathcal{X}_{2}\times\mathcal{X}_{1}$ be a simulation relation of $\Sigma_{2}$ by $\Sigma_{1}$. Then $\mathcal{R}:=\mathcal{S}+\mathcal{T}^{-1}$ is a bisimulation relation between $\Sigma_{1}$ and $\Sigma_{2}$.
	\end{proposition}
	\begin{proof}
		Let $\mathcal{S}$ satisfy (\ref{59}) and let $\mathcal{T}$ satisfy (\ref{59}) with index 1 replaced by 2. Define $\mathcal{R}=\mathcal{S}+\mathcal{T}^{-1}$, then we have properties (\ref{30}a). Similarly, $\mathcal{R}$ satisfies (\ref{30}b,c,d).
	\end{proof}
	\begin{proposition}
		Suppose there exists a simulation of $\Sigma_{1}$ by $\Sigma_{2}$, and a simulation of $\Sigma_{2}$ by $\Sigma_{1}$. Let $\mathcal{S}^{max}\subset\mathcal{X}_{1}\times\mathcal{X}_{2}$ denote the maximal simulation relation of $\Sigma_{1}$ by $\Sigma_{2}$, and $\mathcal{T}^{max}\subset\mathcal{X}_{2}\times\mathcal{X}_{1}$ the maximal simulation relation of $\Sigma_{2}$ by $\Sigma_{1}$. Then $\mathcal{S}^{max}=(\mathcal{T}^{max})^{-1}=\mathcal{R}^{max}$, with $\mathcal{R}^{max}$ the maximal bisimulation relation.
	\end{proposition}
	\begin{proof}
		Analogous to the proof of \cite[Proposition 5.4]{Arjan04a}.
	\end{proof}
	
	Simulation relations appear naturally in the context of {\it abstractions}; see e.g. \cite{Pap03}. Consider the DAE system 
	\begin{equation}
	\label{61}
	\Sigma: \begin{array}{rlllll}
	E\dot{x}&=& Ax+Bu+Gd, & x\in\mathcal{X}, u\in\mathcal{U}, d\in\mathcal{D}, \\[2mm]
	y &=& Cx, & y\in\mathcal{Y},
	\end{array}
	\end{equation}
	together with a surjective linear map $H:\mathcal{X}\rightarrow\mathcal{Z}$, $\mathcal{Z}$ being another linear space, satisfying $\ker H \subset \ker C$.
	This implies that there exist a unique linear map $\bar{C}:\mathcal{Z}\rightarrow\mathcal{Y}$ such that
	\begin{equation}
	\label{62}
	C=\bar{C}H.
	\end{equation}
	Then define the following dynamical system on $\mathcal{Z}$
	\begin{equation}
	\label{63}
	\Sigma :\begin{array}{rllll}
	\bar{E}\dot{z}&=& \bar{A}z+\bar{B}u+\bar{G}d, & z\in\mathcal{Z}, u\in\mathcal{U}, d\in\mathcal{D}, \\[2mm]
	y&=&\bar{C}z, & y\in\mathcal{Y}
	\end{array}
	\end{equation}
	where $H^{+}$ denotes the Moore-Penrose pseudo-inverse of $H$, $\bar{E}:=EH^{+},\bar{A}:= AH^{+},\bar{B}:= B,$ and 
	\[
	\bar{G}:= \left[
	G\vdots E(\ker H)\vdots A(\ker H)
	\right],
	\]			
	is an \textit{abstraction} of $\Sigma$ in the sense that we factor out the part of the state variables $x\in\mathcal{X}$ corresponding to ker $H$. Since $H^{+}z=x+\ker H$, it can be easily proved that $\mathcal{S}:=\{(x,z)\mid z=Hx\}$ is a simulation relation of $\Sigma$ by $\bar{\Sigma}$.  
	
	%%%%%%%%%%%%%%%%%%%%%%%%%%%%%%%%%%%%%%%%%%%%%%%%%%%%%%%%%%%%%%%%%%%%%%%%%%%%%%%%

%%%%%%%%%%%%%%%%%%%%%%%%%%%%%%%%%%%%%%%%%%%%%%%%%%%%%%%%%%%%%%%%%%%%%%%%%%%%%%%%

\section{Conclusions}
In this paper we have defined and studied by methods from geometric control theory the notion of bisimulation relation for general linear differential-algebraic systems, including the special case of DAE systems with regular matrix pencil. Also the one-sided notion of simulation relation related to abstraction has been provided. Avenues for further research include the use of bisimulation relations for model reduction, the consideration of switched DAE systems, as well as the generalization to nonlinear DAE systems.

%%%%%%%%%%%%%%%%%%%%%%%%%%%%%%%%%%%%%%%%%%%%%%%%%%%%%%%%%%%%%%%%%%%%%%%%%%%%%%%%

\section*{Acknowledgment}
The work of the first author is supported by the Directorate General of Resources for Science, Technology and Higher Education, The Ministry of Research, Technology, and Higher Education of Indonesia.

%%%%%%%%%%%%%%%%%%%%%%%%%%%%%%%%%%%%%%%%%%%%%%%%%%%%%%%%%%%%%%%%%%%%%%%%%%%%%%%%

\end{document}